\documentclass[10pt]{amsart}
\usepackage{a4}
\usepackage{mathtools}
\usepackage[bookmarks=true,bookmarksopen=true,colorlinks=true,linkcolor=blue,citecolor=blue,urlcolor=blue]{hyperref}

\usepackage{amsthm}
\usepackage{url}
\usepackage{xcolor}
\usepackage{csquotes}

\usepackage{microtype}

\theoremstyle{plain}
\newtheorem{theorem}{Theorem}[section]

\newtheorem{corollary}[theorem]{Corollary}
\theoremstyle{remark}
\newtheorem{remark}[theorem]{Remark}

\DeclareMathOperator{\re}{Re}

\title[On an integral identity]{On an integral identity}

\author[A.~Bostan]{Alin Bostan}
\address[Alin Bostan]{Inria and Universit\'e Paris-Saclay, 1 rue Honor\'e d'Estienne d'Orves, 91120 Palaiseau, France}
\email{alin.bostan@inria.fr}
\author[F.~Chamizo]{Fernando Chamizo}
\address[Fernando Chamizo]{Universidad Aut\'onoma de Madrid and ICMAT, Ciudad universitaria de Cantoblanco, 28049 Madrid, Spain}
\email{fernando.chamizo@uam.es}
\author[M. P.-Sundqvist]{Mikael P. Sundqvist}
\address[Mikael Persson Sundqvist]{Lund University, Department of Mathematical 
Sciences, Box 118, 221\ 00 Lund, Sweden.}
\email{mikael.persson\_sundqvist@math.lth.se}

\begin{document}

\makeatletter
\def\author@andify{%
  \nxandlist {\unskip ,\penalty-1 \space\ignorespaces}%
    {\unskip {} \@@and~}%
    {\unskip \penalty-2 \space \@@and~}%
}
\makeatother

\begin{abstract} 

We give three elementary proofs of a nice equality of definite integrals,
which arises from the theory of bivariate hypergeometric functions, and
has connections with irrationality proofs in number theory. We furthermore
provide a generalization together with an equally elementary proof and discuss
some consequences.

\end{abstract}

\maketitle

\section*{Introduction}\label{sec:intro}

The following infinite family of equalities between definite integrals was
proven by S. B. Ekhad, D. Zeilberger and W.~Zudilin in~\cite{EkZeZu}, using
the Almkvist--Zeilberger \emph{creative telescoping algorithm}~\cite{AlZe90}
for symbolic integration:
\begin{equation}\label{eq:EZZ}
\int_{0}^{1}\!{\frac {{x}^{n} \left( 1-x \right) ^{n}}{ \left(  \left( x+a 
\right)  \left( x+b \right)  \right) ^{n+1}}}\,{\mathrm d}x
=
\int_{0}^{1}\!{\frac {{x}^{n} \left( 1-x \right) ^{n}}{ \left(  \left( a-b 
\right) x+ \left( a+1 \right) b \right) ^{n+1}}}\,{\mathrm d}x,
\end{equation}
for any reals $a > b > 0$ and any nonnegative integer $n$.

As pointed out in~\cite{EkZeZu}, these integrals \enquote{are not taken from a pool
of no-one-cares analytic creatures}: they are related to rational
approximations to some logarithmic values \cite{AlRo80} and a trained eye
could recognize in identity~\eqref{eq:EZZ} a particular case of a known
relation for Appell's bivariate hypergeometric function 
(see \S\ref{sec:consequences} below).

The proof provided in~\cite{EkZeZu} is, without any doubt, elementary. It
requires some clever (and at first sight, magic) auxiliary rational functions
coming from the silicon guts of the first author of~\cite{EkZeZu}. Of course,
this is not objectable at all, and we are ourselves convinced that
computer-assisted proofs are an increasingly important trend in mathematics.
At the same time we think that, \emph{pour l'honneur de l'esprit humain}, it
is of some interest to offer a proof that a freshperson could not only follow
but also create. 

In Sections~\ref{sec:first}--\ref{sec:third} of this note we provide three 
elementary proofs of the identity~\eqref{eq:EZZ}. 
We also generalize the identity. 
In Section~\ref{sec:consequences} we discuss some direct
consequences, emphasizing the relation of the identities 
with known identities for hypergeometric functions. We also give a couple of
combinatorical identities, and finally conclude that the
Legendre polynomials are eigenfunctions of a certain differential operator.

\section{First proof: using a rational change of variables}\label{sec:first}

Note that the natural range to assure the convergence of the integrals is
$a,b>0$ and $n\in\mathbb{R}_{>-1}$. Let us demonstrate firstly that an utterly
simple change of variables proves the equality in this extended range and in
fact an even more general equality.

\begin{proof}[{Proof of \eqref{eq:EZZ} for $a,b>0$ and $n\in\mathbb{R}_{>-1}$}]
With the rational change of variables
\[
x=\frac{b(1-u)}{b+u},
\quad\text{one gets}\quad
{\mathrm d}x=-\frac{b(1+b)}{(b+u)^2}\,{\mathrm d}u.
\]
Note that it is a promising change because it takes $0$, $1$, $b(a+1)/(b-a)$ and $\infty$, which are the singularities of the second integrand, into $1$, $0$, $-a$ and $-b$, respectively, which are singularities of the first integrand.
The interval $[0,1]$ is preserved. Thus, a direct calculation yields
\[
\begin{aligned}
\int_0^1\frac{x^n(1-x)^n}{\bigl((x+a)(x+b)\bigr)^{n+1}}\,{\mathrm d}x
&=\int_0^1\frac{\Bigl(\frac{b(1-u)}{b+u}\Bigr)^n\Bigl(1-\frac{b(1-u)}{b+u}\Bigr)^n}{\Bigl[\Bigl(\frac{b(1-u)}{b+u}+a\Bigr)\Bigl(\frac{b(1-u)}{b+u}+b\Bigr)\Bigr]^{n+1}}\frac{b(1+b)}{(b+u)^2}\,{\mathrm d}u\\
&=
\int_0^1\frac{(1-u)^nu^n b^{n+1}(1+b)^{n+1}\,{\mathrm d}u}{\bigl[\bigl(b(1-u)+a(b+u)\bigr)\bigl(b(1-u)+b(b+u)\bigr)\bigr]^{n+1}}\\
&=
\int_0^1\frac{u^n(1-u)^n}{\bigl((a-b)u+(a+1)b\bigr)^{n+1}}\,{\mathrm d}u,
\end{aligned}
\]
which proves~\eqref{eq:EZZ}.
\end{proof}

Actually, the same change of variables also proves a generalization of~\eqref{eq:EZZ}:
\begin{theorem}\label{th:main}
If $a$, $b>0$, $k$, $n\in\mathbb{R}$ and $s$, $\ell \in\mathbb{R}_{>-1}$, then
\begin{equation}\label{eq:4EZZ}
\int_0^1 \frac{x^{\ell} (1-x)^{s}}{\left(x+a\right)^{k+1} \, \left(x+b \right)^{n+1}} \, {\mathrm d}x
=
\frac{(b+1)^{s-n}}{b^{n-\ell}}
\int_0^1 \frac{x^{s} \, (1-x)^{\ell} \, (x+b)^{n+k -\ell-s}}{\left((a-b) \, x + (a+1)b \right)^{k+1}} \, {\mathrm d}x.
\end{equation}
\end{theorem}
We will come back to and draw some conclusions from this more general identity in Section~\ref{sec:consequences}.


\section{Second proof: using indefinite integration}\label{sec:second}

We still consider the extended convergence range $a,b>0$ but now $n$ is a nonnegative integer. The following proof is based on the generating functions \cite{wilf} of the sequence of integrals when $n$ varies. 

\begin{proof}[{Proof of \eqref{eq:EZZ} for $a,b>0$ and $n\in\mathbb{Z}_{\ge 0}$}]
 Let $I_1$ and $I_2$ be the generating functions of each  side of~\eqref{eq:EZZ} i.e., multiplying by $t^n$ and summing from $n=0$ to $\infty$. They clearly converge uniformly for small values of $t$ and we have $I_j(t)=\int_0^1\,{\mathrm d}x/P_j(x,t)$ where
 \begin{equation}\label{p1p2}
  P_1(x,t)=(x+a)(x+b)-tx(1-x)
  \quad\text{and}\quad
  P_2(x,t)=(a-b)x+(a+1)b-tx(1-x).
 \end{equation}
 If $P(x)=Ax^2+Bx+C$ has no zero in $[0,1]$, and $\Delta=B^2-4AC>0$, then, using standard integration techniques,
 \[
  \int_0^1
  \frac{{\mathrm d}x}{P(x)}
  =
  \frac{1}{r}
  \log\Big(
  \frac{B+2C+r}{B+2C-r}
  \Big)
  \qquad\text{with}\quad
  r=\sqrt{\Delta}.
 \]
 Note that for $t$ small enough $P_1$ and $P_2$ fulfill these conditions. A calculation shows that both polynomials have the same discriminant $\Delta$ and the same values of $B+2C$. Hence $I_1(t)=I_2(t)$ in some interval containing the origin and then the integrals in~\eqref{eq:EZZ}, which are their Taylor coefficients, are equal. 
\end{proof}

\section{Third proof: creative telescoping}\label{sec:third}
We cannot resist the temptation to offer a third proof, in the
spirit of the one in~\cite{EkZeZu}, but based on a different kind of
 \enquote{creative telescoping}. The starting point is the same as in Section~\ref{sec:second}, namely that the family of identities~\eqref{eq:EZZ} is equivalent to the fact
that the two integrals $I_1(t)$ and~$I_2(t)$ between $x=0$ and $x=1$ of the 
rational functions $F_1=1/P_1$ and $F_2=1/P_2$, with $P_j$ as in \eqref{p1p2},
are equal.

Creative telescoping (this time in the classical \enquote{differential-differential} setting) shows that $F_1$ and $F_2$ satisfy the equalities
\[
 \left( t-2\,ab-a-b \right) F_j+ \left( {t}^{2}-2\,t \left( 2\,ab+a+b \right) + \left( a-b \right) ^{2} \right) {\frac {\partial F_j}{\partial t}}
+ 
{\frac {\partial}{\partial x}} \left( F_j  R_j \right) = 0
\]
where $R_1(t,x)$ and $R_2(t,x)$ are the rational functions
\[
R_1(t,x) =  \bigl(( a+b+t+2) x+2ab+a+b-t \bigr) x
\]
and
\[
R_2(t,x) = {\frac { \bigl( ( 2ab+a+b ) t- (a-b)^{2} \bigr)x^2+b ( a+1)(2ab+a+b-t)}{t+b-a}}.
\]
Hence, by integration between $x=0$ and $x=1$, one obtains that both $I_1$ and $I_2$ are 
solutions of the differential equation
\[
(t-2ab-a-b) I(t) + \bigl( t^2-2t ( 2ab+a+b) + (a-b)^2 \bigr)I'(t) +2 =0
\]
Therefore $I_1=I_2$ by Cauchy's theorem, since $I_1 - I_2$  is the solution of a differential
equation of order $1$ with leading term non-vanishing at $t=0$, and its
evaluation at $t=0$ is zero, as $I_1(0) = I_2(0) = {\frac {1}{a-b}\ln  \left( {\frac {a \left( b+1 \right) }{ \left( a+1 \right) b}} \right) }
$. \hfill $\square$

\section{Some consequences}\label{sec:consequences}

\subsection{Appell's identity}

We will show that identity~\eqref{eq:4EZZ} from our Theorem~\ref{th:main}
contains, as a particular case, a classical hypergeometric
identity due to Appell.
The relevant definitions to state this identity are the classical Gauss
hypergeometric function $_2F _1(\alpha, \beta;\gamma;t)$ and the Appell
bivariate hypergeometric function $F_1(\alpha;\beta,\beta';\gamma;x,y)$,
given respectively by
\[
\sum_{n=0}^\infty \frac{(\alpha)_n(\beta)_n}{(\gamma)_n} \, \frac
{t^n} {n!}
\qquad\text{and}\qquad
\sum_{m,n\geq 0} 
\frac{(\alpha)_{m+n} \, (\beta)_{m} \, (\beta')_{n}}{(\gamma)_{m+n}}
\, \frac{x^m y^n}{m! n!},
\]
where $(a)_n$ denotes the rising factorial $a(a+1)\cdots(a+n-1)$ for 
$n\in\mathbb{N}$ and it is assumed $|t|,|x|,|y|<1$ to assure the convergence.

It is very classical that these hypergeometric functions admit the following integral
representations, which hold as soon as $\beta, \beta'>0$ and $\gamma :=
\beta+\beta' > \alpha >0$:
\begin{equation}\label{eq:Euler}
_2F _1 \left(\alpha, \beta; \gamma; z\right)
=
\frac{\Gamma(\gamma)}{\Gamma(\beta) \Gamma(\beta')} 
\int_0^1 \frac{t^{\beta-1} (1-t)^{\beta'-1}} {\left(1- t z \right)^{\alpha}}\, {\mathrm d}t  
\end{equation}
and
\begin{equation}\label{eq:Picard}
 F_1(\alpha;\beta,\beta';\gamma;x,y) = 
\frac{\Gamma(\gamma)}{\Gamma(\alpha) \Gamma(\gamma-\alpha)} 
\int_0^1  \frac{ t^{\alpha-1} (1-t)^{\gamma-\alpha-1} }
 { \left(1- t x \right)^{\beta} \,
 \left(1- t y \right)^{\beta'}}\, {\mathrm d}t. 
\end{equation}
Equation~\eqref{eq:Euler} is due to Euler~\cite[Th.~2.2.1]{AnAsRo99}, and~\eqref{eq:Picard} to
Picard~\cite{Picard1881}, see also~\cite[Eq.~(9)]{Appell1925}.

With these notations, 
we are able to state the following hypergeometric function identity,
which appears
on page 8 of Appell's classical memoir~\cite{Appell1925}:
\begin{corollary}[Appell's identity]\label{c1}
If $\alpha$, $\beta$, $\beta'>0$, $\beta+\beta'>\alpha$, $|x|<1$ and $|y|<1$ then
\begin{equation}\label{eq:Appell}
  _2F _1 \left(\alpha, \beta;\beta+\beta';\frac{y-x}{y-1}\right) 
 =
 (1-y)^{\alpha} \, 
 F_1(\alpha;\beta,\beta';\beta+\beta';x,y).
\end{equation}
Equivalently, in terms of integrals:
\begin{equation}\label{eq:gen}
\int_0^1 \frac{t^{\beta-1} \, (1-t)^{\beta'-1}}{\left((y-x) \, t + 1-y \right)^{\alpha}} \, {\mathrm d}t
=
\frac{\Gamma(\beta) \Gamma(\beta')}{\Gamma(\alpha) \Gamma(\beta+\beta'-\alpha)}
\int_0^1 \frac{t^{\alpha-1} (1-t)^{\beta+\beta'-\alpha-1}}{\left(1- t x \right)^{\beta} \, \left(1- t y \right)^{\beta'}} \, {\mathrm d}t.  
\end{equation}
\end{corollary}

\begin{remark}
Identity~\eqref{eq:EZZ} is a particular case of~\eqref{eq:gen}, with $\alpha=\beta=\beta'=n+1$.
\end{remark}

\begin{proof}[Proof of Corollary~\ref{c1}]
If we set $n=\ell+s-k$ in~\eqref{eq:4EZZ}, then the identity becomes
\begin{equation}\label{eq:5EZZ}
\int_0^1 \frac{x^{\ell} (1-x)^{s}}{\left(x+a\right)^{k+1} \, \left(x+b \right)^{\ell+s-k+1}} \, {\mathrm d}x
=
\frac{(b+1)^{k-\ell}}{b^{s-k}}
	\int_0^1 \frac{x^{s} \, (1-x)^{\ell}}{\left((a-b) \, x + (a+1)b \right)^{k+1}} \, {\mathrm d}x.
\end{equation}
Next we replace the integration variable $x$ with $1-t$ in the first integral 
and by $t$ in the second integral, then $a$ with $-1/x$ and $b$ with $-1/y$, 
to deduce the following equivalent form for $x,y<0$:
\[
\int_0^1 \frac{t^{\ell} \, (1-t)^{s}}{\bigl(1 - \frac{y-x}{y-1} \, t \bigr)^{k+1} } \, {\mathrm d}t
=
(1-y)^{\ell+1}
\int_0^1 \frac{t^{\ell} (1-t)^{s}}{(1-tx)^{k+1} \, (1-ty)^{\ell+s-k+1}} \, {\mathrm d}t,
\]
that can be analytically continued to the values of $x$ and $y$ as in the statement.
Inserting $\ell = \alpha - 1$, $k = \beta - 1$ and $s =  \beta + \beta' -\alpha -1$,
using the integral representations~\eqref{eq:Euler} and~\eqref{eq:Picard} 
and the obvious symmetry $_2F _1(\alpha, \beta;\gamma;t)= {_2F _1}(\beta, \alpha;\gamma;t)$, we deduce \eqref{eq:Appell}.
\end{proof}


The close relationship between \eqref{eq:5EZZ} and the univariate and bivariate 
hypergeometric functions is shown in the previous proof of Corollary~\ref{c1}. 
Taking into account that~\eqref{eq:gen} is a formulation of this result not 
involving any hypergeometric function, it appears as a natural problem to 
provide a more direct proof of \eqref{eq:gen}. We present an independent proof 
involving basic real and complex analysis.

\begin{proof}[Alternative proof of \eqref{eq:gen}]
We write $\gamma=\beta+\beta'$ as before. Changing $t\mapsto 1-t$ in the first integral of~\eqref{eq:gen} and multiplying the identity by 
 $\Gamma(\alpha)\Gamma(\gamma-\alpha)$, we want to prove that the functions
 \[
  G_1(x,y)
  =
  \Gamma(\alpha) \Gamma(\gamma-\alpha)
 \int_0^1 \frac{t^{\gamma-\beta-1}  (1-t)^{\beta-1}}{\left(  1-ty-(1-t)x \right)^{\alpha}} \, {\mathrm d}t
 \]
 and
 \[
  G_2(x,y)
  =
 \Gamma(\beta) \Gamma(\gamma-\beta)
 \int_0^1 \frac{t^{\alpha-1} (1-t)^{\gamma-\alpha-1}}{\left(1-tx  \right)^{\beta} \, \left(1- ty\right)^{\gamma-\beta}} \, {\mathrm d}t
 \]
 coincide. Both are analytic functions of $x$ and $y$ on $\re(x),\re(y)<1/2$ 
(for instance by Morera's theorem applied separately in $x$ and $y$). 
Then it is enough to prove
 \[
  \frac{\partial^{n+m} G_1}{\partial x^n\partial y^m}
  (0,0)
  =
  \frac{\partial^{n+m} G_2}{\partial x^n\partial y^m}
  (0,0)
  \qquad\text{for every }m,n\in\mathbb{Z}_{\ge 0},
 \]
 because in this case their Taylor coefficients coincide. 
 Noting that the $k$-th derivative of $(1-x)^{-\delta}$
 is $\frac{\Gamma(\delta+k)}{\Gamma(\delta)}(1-x)^{-\delta-k}$, this is the same as
 \[
   \Gamma(\gamma-\alpha)
  \Gamma(\alpha+m+n)
  B(\gamma+m-\beta,\beta+n)
 =
 \Gamma(\beta+n)
 \Gamma(\gamma+m-\beta)
 B(\alpha+m+n,\gamma-\alpha)
 \]
 with $B(p,q)=\int_0^1 t^{p-1}(1-t)^{q-1}\, {\mathrm d}t$. 
Using the well-known elementary evaluation 
$B(p,q)=\Gamma(p)\Gamma(q)/\Gamma(p+q)$, the proof is complete.
\end{proof}

\subsection{Combinatorial identities and Legendre polynomials}

The equality~\eqref{eq:EZZ} and its generalizations can be interpreted in
combinatorial terms. If we in~\eqref{eq:4EZZ} set $a=b>0$, 
$\ell=s=n$ and replace $k$ by $k-1$ then we get
\begin{equation}\label{eq:chv}
  \int_0^1
  \frac{x^n(1-x)^n}{(x+b)^{n+k+1}}
  \, {\mathrm d}x
  =
  \frac{1}{b^k(b+1)^k}
  \int_0^1
  \frac{x^n(1-x)^n}{(x+b)^{n-k+1}}
  \, {\mathrm d}x.
\end{equation}

\begin{corollary}\label{c2}
 The functions $f_n(x)=x^n(1+x)^n$ with $n$ a nonnegative integer satisfy
 \[
  f_k 
  f_n^{(n+k)}
  =
  \frac{(n+k)!}{(n-k)!}
  f_n^{(n-k)}
  \qquad\text{for $0\le k\le n$.}
 \]
\end{corollary}

\begin{proof}
 By Taylor expansion at $x=-b$ we have
 \[
  x^n(1-x)^n
  =
  (-1)^n f_n(-x)
  =
  \sum_{m=0}^{2n}
  c_m(x+b)^m
  \qquad\text{with}\quad
  c_m
  =
  (-1)^{n+m}\frac{f_n^{(m)}(b)}{m!}.
 \]
 If we substitute this in \eqref{eq:chv} a term  with $\log(b+1)-\log b$ appears 
in the LHS for $m=n+k$ (the rest are rationals) and in the integral of the RHS 
for $m=n-k$. Hence $c_{n+k}=c_{n-k}/\big(b(b+1)\big)^k$. 
\end{proof}

Essentially comparing coefficients one gets a triple binomial identity: 
\begin{corollary}\label{c3}
  For each $k$, $\ell$, $n$ nonnegative integers with $\ell,k\le n$
  \[
   \sum_{m=0}^k
   \binom{k}{m}
   \binom{n}{m+\ell}
   \binom{2m+2\ell}{k+n}
   =
   \binom{n}{\ell}
   \binom{2\ell}{n-k}.
  \]
\end{corollary}

\begin{remark}
This hypergeometric identity can be rewritten as
\[
_3 F_2 \Bigl(-k,\ell-n,\ell+\frac12;\,\ell-\frac{k+n}{2}+1,
\frac{1-n-k}{2}+\ell;\,1 \Bigr)
= \frac{\binom{2\ell}{n-k}}{\binom{2\ell}{n+k}},
\]
so it is a particular case of the Pfaff--Saalsch\"utz identity~\cite[\S2.3.1]{Slater}
for the evaluation at 1 of well-poised $_3F_2$'s.
Alternatively, it can be automatically obtained using Zeilberger's 
creative telescoping algorithm~\cite{PeWiZe}.
\end{remark}

\begin{proof}
In Corollary~\ref{c2} change the variable $x\mapsto (x-1)/2$ to obtain, 
clearing denominators,
 \begin{equation}\label{eq:di}
     (n-k)!(x^2-1)^k
   \frac{{\mathrm d}^{n+k}}{{\mathrm d}x^{n+k}}
   (x^2-1)^n
   =
   (n+k)!
   \frac{{\mathrm d}^{n-k}}{{\mathrm d}x^{n-k}}
   (x^2-1)^n.
 \end{equation}
We rewrite the left-hand side as
\[
  (n-k)!
  \sum_{\ell_1}
  (-1)^{k-\ell_1}
  \binom{k}{k-\ell_1}
  x^{2\ell_1}
  \cdot
  (n+k)!
  \sum_{\ell_2}
  (-1)^{n-\ell_2}
  \binom{n}{\ell_2}
  \binom{2\ell_2}{n+k}
  x^{2\ell_2-n-k},
 \]
and the right-hand side as
 \[
  (n+k)!(n-k)!
  \sum_{\ell}
  (-1)^{n-\ell}
  \binom{n}{\ell}
  \binom{2\ell}{n-k}
  x^{2\ell-n+k}.
\]
Comparing coefficients and renaming $\ell_1=k-m$, $\ell_2=m+\ell$ yields the 
result. 
\end{proof}

\begin{remark}
It turns out that~\eqref{eq:di} is a known identity expressing a symmetry
of the associated Legendre polynomials. This can be obtained from Rodrigues'
formula and from the general Legendre equation, see~\cite{westra} for
a simple elementary proof. The one obtained here (Corollary~\ref{c2}) is
competitively simple.
\end{remark}

The definition of the Legendre polynomials assures that they are eigenfunctions 
of the differential operator 
$\frac{{\mathrm d}}{{\mathrm d}x}(x^2-1)\frac{{\mathrm d}}{{\mathrm d}x}$. 
Although $\frac{{\mathrm d}}{{\mathrm d}x}$ and $(x^2-1)$ do not commute and 
apparently there is not a simple formula for the commutator of their powers, 
Legendre polynomials are also eigenfunctions of a simple operator composed by 
powers of these operators. 

\begin{corollary}\label{c4}
  Let $P_n$ be the $n$-th Legendre polynomial. Then for $0\le k\le n$
  \[
   L[P_n]
   =
  \frac{(n+k)!}{(n-k)!}P_n
  \qquad\text{where}\quad
  L
  =
  \frac{{\mathrm d}^k}{{\mathrm d}x^k}(x^2-1)^k\frac{{\mathrm d}^k}{{\mathrm d}x^k}.
  \]
\end{corollary}

\begin{proof}
 Recalling that $P_n$ is proportional to 
 $\frac{{\mathrm d}^{n}}{{\mathrm d}x^{n}}(x^2-1)^n$ (Rodrigues' formula) 
this follows taking the $k$-derivative of~\eqref{eq:di}. 
\end{proof}

Actually, one could prove Corollary~\ref{c2} from Corollary~\ref{c4} by 
repeated integration, and noting that both sides in Corollary~\ref{c2} are 
divisible by $x^k$.


\end{document}